\documentclass[11pt]{amsart}
\usepackage{amssymb, amsmath, amsthm, mathrsfs, amsopn}
\usepackage{amsrefs}
\usepackage{graphicx}
\usepackage{color}
\usepackage[a4paper, centering]{geometry}
\def\pscal#1#2{\left\langle #1\,,\, #2 \right\rangle}
\def\nuint{\widetilde{\nu}}
\def\preciso#1{\widetilde #1}
\newcommand*{\chiut}[1][]{\chi^{#1}_{\{u > t\}}}
\newcommand*{\jump}[1]{\Theta_{#1}}

\newcommand{\A}{\boldsymbol{A}}
\newcommand{\B}{\boldsymbol B}
\newcommand{\bsmall}{\boldsymbol b}
\newcommand{\F}{\mathcal{F}}

\newcommand{\mean}[1]{\,-\hskip-1.1em\int_{#1}} %media integrale displayed
\newcommand{\Leb}{\mathcal{L}} % Misura di Lebesgue

\def\R{\mathbb{R}}
\def\RN{\mathbb{R^N}}

\def\N{\mathbb{N}}

\def\Om{\Omega}

\def\DM{{\mathcal{DM}^{\infty}}}
\newcommand*{\DMloc}[1][\Omega]{\mathcal{DM}^{\infty}_{\rm loc}{(#1)}}
\def\sfera{{\mathchoice
{\setbox0=\hbox{$\displaystyle     \rm S$}\hbox{\raise0.5\ht0\hbox
to0pt{\kern0.35\wd0\vrule height0.45\ht0\hss}\hbox
to0pt{\kern0.55\wd0\vrule height0.5\ht0\hss}\box0}}
{\setbox0=\hbox{$\textstyle        \rm S$}\hbox{\raise0.5\ht0\hbox
to0pt{\kern0.35\wd0\vrule height0.45\ht0\hss}\hbox
to0pt{\kern0.55\wd0\vrule height0.5\ht0\hss}\box0}}
{\setbox0=\hbox{$\scriptstyle      \rm S$}\hbox{\raise0.5\ht0\hbox
to0pt{\kern0.35\wd0\vrule height0.45\ht0\hss}\raise0.05\ht0\hbox
to0pt{\kern0.5\wd0\vrule height0.45\ht0\hss}\box0}}
{\setbox0=\hbox{$\scriptscriptstyle\rm S$}\hbox{\raise0.5\ht0\hbox
to0pt{\kern0.4\wd0\vrule height0.45\ht0\hss}\raise0.05\ht0\hbox
to0pt{\kern0.55\wd0\vrule height0.45\ht0\hss}\box0}}}}
\def\q{{\mathchoice {\setbox0=\hbox{$\displaystyle\rm
Q$}\hbox{\raise
0.15\ht0\hbox to0pt{\kern0.4\wd0\vrule height0.8\ht0\hss}\box0}}
{\setbox0=\hbox{$\textstyle\rm Q$}\hbox{\raise
0.15\ht0\hbox to0pt{\kern0.4\wd0\vrule height0.8\ht0\hss}\box0}}
{\setbox0=\hbox{$\scriptstyle\rm Q$}\hbox{\raise
0.15\ht0\hbox to0pt{\kern0.4\wd0\vrule height0.7\ht0\hss}\box0}}
{\setbox0=\hbox{$\scriptscriptstyle\rm Q$}\hbox{\raise
0.15\ht0\hbox to0pt{\kern0.4\wd0\vrule height0.7\ht0\hss}\box0}}}}

\newcommand*{\BVLloc}[1][\Omega]{BV_{{\rm loc}}(#1)\cap L^{\infty}_{{\rm loc}}{(#1)}}
\newcommand*{\BVL}[1][\Omega]{BV(#1)\cap L^{\infty}{(#1)}}
\def\rel #1{{\overline #1}}
\DeclareMathOperator{\Tr}{Tr}
\newcommand*{\Trace}[3][\pm]{\Tr^{#1}(#2, #3)}
\newcommand*{\Trp}[2]{\Trace[+]{#1}{#2}}
\newcommand*{\Trm}[2]{\Trace[-]{#1}{#2}}

\def\vv{{\bf v}}

\def\F{{\mathcal F}}

\def\preciso #1{\widetilde #1}
\def\L{{\mathcal L}}
\newcommand{\LLU}{{\mathcal L}^1}
\def\M{{\mathcal M}}
\def\pallaunit{{\omega^{}_N}}
\def\RN{{\R^N}}

\DeclareMathOperator{\Div}{div}
\DeclareMathOperator{\spt}{spt}
\newcommand{\medint}{-\kern  -,375cm\int}
\newcommand{\medintinrigo}{-\kern  -,315cm\int}

 \newcommand{\hh}{{\mathcal H}^{N-1}}

\def\H{{\mathcal H}}
\newcommand{\LLN}{{\mathcal L}^N}
\newcommand{\Haus}[1]{{\mathcal H}^{#1}} % Misura di Hausdorff

\newcommand{\res}{\mathop{\hbox{\vrule height 7pt width .5pt depth 0pt
\vrule height .5pt width 6pt depth 0pt}}\nolimits} % macro per la restrizione
\newcommand{\ristretto}{\mathop{\hbox{\vrule height 7pt width .5pt depth 0pt
\vrule height .5pt width 6pt depth 0pt}}\nolimits} % macro per la restrizione

\def\pscal#1#2{\left\langle #1\,,\, #2 \right\rangle}

\DeclareMathOperator{\sign}{sign}
\def\ut{\widetilde{u}}

\newcommand*{\LDM}[1][\Omega]{\mathcal{DL}^1(#1)}
\newcommand*{\LDMloc}[1][\Omega]{\mathcal{DL}^1_{\text{loc}}(#1)}
\DeclareMathOperator{\cyl}{Cyl}
\newcommand*{\Cyl}[3]{\cyl(#1, #2; #3)}

\newtheorem{definition}{Definition}[section]
\newtheorem{lemma}[definition]{Lemma}
\newtheorem{theorem}[definition]{Theorem}
\newtheorem{proposition}[definition]{Proposition}
\newtheorem{corollary}[definition]{Corollary}
\theoremstyle{remark}
\newtheorem{remark}[definition]{Remark}

\makeatletter
\def\@settitle{\begin{center}%
		\baselineskip14\p@\relax
		\bfseries
		\uppercasenonmath\@title
		\@title
		\ifx\@subtitle\@empty\else
		\\[5ex]%\@subtitle
		\normalsize\mdseries\@subtitle
		\fi
	\end{center}%
}
\def\subtitle#1{\gdef\@subtitle{#1}}
\def\@subtitle{}
\makeatother
\begin{document}
\title[On the variational nature of the Anzellotti pairing]
{On the variational nature of the Anzellotti pairing}

\author[G.~Crasta]{Graziano Crasta}
\address{Dipartimento di Matematica ``G.\ Castelnuovo'', Univ.\ di Roma I\\
	P.le A.\ Moro 2 -- I-00185 Roma (Italy)}
\email{graziano.crasta@uniroma1.it}
\author[V.~De Cicco]{Virginia De Cicco}
\address{Dipartimento di Scienze di Base  e Applicate per l'Ingegneria, Univ.\ di Roma I\\
	Via A.\ Scarpa 10 -- I-00185 Roma (Italy)}
\email{virginia.decicco@uniroma1.it}


% 28: Measure and integration
% 28B05: Vector-valued set functions, measures and integrals
%
% 46: Functional analysis
% 46G10: Vector-valued measures and integration
%
% 26: Real functions
% 26B: Functions of several variables
% 26B30: Absolutely continuous functions, functions of bounded variation

%\date{\today}
\date{October 7, 2024}

\begin{abstract}

In this paper we prove that the 
Anzellotti pairing can be regarded as a relaxed functional with respect to the weak$^\star$ convergence in the space $BV$ of functions of bounded variation.
The crucial tool is a preliminary integral representation of this pairing by means of suitable cylindrical averages.

\end{abstract}

\maketitle

\section{Introduction}
A classical problem in Calculus of Variations is to find minimal assumptions assuring the lower semicontinuity with respect to a suitable convergence of  integral functionals of the type
\begin{equation*}
%\label{functional}
J(u)=\int_\Omega f(x,u,\nabla u)\,dx,
\end{equation*}
where $\Omega$ is an open subset of $\R^N$ and $u\colon\Omega\to\R$ belongs to a given space of weakly differentiable functions. 
With this problem in mind,
it is well known that, if the integrand $f(x, s,\xi)$ admits a linear growth with respect to the gradient variable $\xi$
when $|\xi| \to +\infty$, the natural  functional framework 
is the space $BV$ of functions of bounded variation.

A classical result in this direction has been proved
by Dal Maso in \cite{DM1980}.
More precisely,
assuming that the integrand $f(x, s,\xi)$ is coercive, continuous and convex in the last variable, he introduced a proper lower semicontinuous extension of $J$ from $W^{1,1}$ to $BV$ and proved that it coincides with the integral representation of the relaxed functional of $J$.
If we drop the coercivity assumption on $f$, the task of studying the lower semicontinuity and of finding the relaxation of $J$ is highly non-trivial and requires some additional regularity assumption on $f$ in the $x$ variable (see for instance \cite{FL,dcl,DFV,DFV1,ADCF1,ADCF2,BouDM}). 

%We remark that,
%in the vectorial case (i.e., when $u$ is $\R^m$-valued),
%the lower semicontinuity and the relaxation have been studied by many authors
%under the quasi-convexity assumption with respect to the gradient variable
%(see the fundamental paper \cite{FM} and the recent works \cite{RS,RS1}).
%%in the vectorial case, starting from the fundamental paper \cite{FM}, the lower semicontinuity and the relaxation has been studied by many authors until the recent works \cite{RS,RS1}, under the quasi-convexity assumption with respect to the gradient variable. 
%In all the mentioned papers, a coercivity (or partial coercivity)
%assumption is required. 
%It is well known that, in this setting,
%the relaxation result can be obtained requiring only the bare
%measurability in the space variable $x$.
%The coercivity assumption is strictly related to the choice of the weak$^*$ convergence, since under this assumption every sequence that converges in $L^1$ 
%converges also in the weak$^*$ topology.
%On the other hand, 
%although our results are confined to the scalar case,
%this improvement by itself (the $L^1$ versus the weak$^*$ relaxation) is crucial in order to apply the results to problems involving very degenerate functionals, like the ones that arise for instance in the weighted least gradient problem. 
%Furthermore, due to the lack of coercivity, 
%the optimal finiteness domain of the relaxed functional could be larger than $BV$
%%it is not possible to fix a priori the finiteness space of the relaxed functional, which could be larger than $BV$ 
%(see \cite{CDS}).
%

The aim of this paper is to investigate 
the possibility of new progress
in this area, by considering the model cases
\begin{gather*}
\mathcal{F}_\varphi(u)=\int_\Omega \varphi\ \bsmall(x,u)\cdot\nabla u\,dx\,,
\quad 
\mathcal{F}(u)=\int_\Omega |\bsmall(x,u)\cdot\nabla u|\,dx\,,
\quad
\varphi\in C^1_c(\Omega)\,,
\
u \in W^{1,1}(\Omega),
\end{gather*}
where the vector field $\bsmall\colon\Omega\times\R\to\R^N$
satisfies the assumptions listed in Section~\ref{ss:assumptions}.
To the authors' knowledge, previous results in literature (see e.g.\ \cite{RS,RS1}) 
are not applicable to $\mathcal{F}_\varphi$ and $\mathcal{F}$, being these functional strongly degenerate.

We remark that
all the results presented in this paper are new also 
in the case of a vector field $\bsmall$ 
independent of $u$. 

The lower semicontinuity of these functionals with respect to the $L^1$ convergence has been established in \cite{dcl} in the Sobolev space $W^{1,1}$ by requiring a very weak regularity assumption, i.e.\ that the divergence of the vector field $\bsmall(x,s)$ with respect to $x$ is an $L^1$ function. 
Our aim is to extend this result to the space $BV$, by considering a relaxed functional defined by an abstract relaxation procedure.

More precisely, for every 
fixed function $\varphi\in C^1_c(\Omega)$ 
and every open set $A\subseteq\Omega$, 
let us consider the functional $F_\varphi(\cdot, A):\BVLloc\to\R\cup\{+\infty\}$ defined by
\begin{equation}
\label{f:Fphi}
F_\varphi(u, A):=
\begin{cases}
\displaystyle\int_A \varphi\  \bsmall(x,u)\cdot\nabla u\,dx
& \text{if $u\in W^{1,1}_{{\rm loc}}(\Omega)$}\cap L_{{\rm loc}}^\infty(\Omega),
\\
+\infty
,
& \text{if}\ u\in (BV_{{\rm loc}}(\Omega)\
\setminus W_{{\rm loc}}^{1,1}(\Omega))\cap L_{{\rm loc}}^\infty(\Omega)\,,
\end{cases}
\end{equation}
and the associated relaxed functional
\begin{equation}
\label{relax1}
\rel{F_\varphi}(u, A)\!:=\inf\left\{\liminf_{n\to +\infty}F_\varphi(u_n, A)\!:\!
u_n\in W^{1,1}_{\rm loc}(\Omega)\cap L_{{\rm loc}}^\infty(\Omega), u_n\to u,
\ %{\rm weak}^*{ \rm \ in\  } 
\text{$w^*$--}BV_{\rm loc}(\Omega)
\right\}\!.
\end{equation}
As it is customary,
this relaxed functional can be characterized as the greatest lower semicontinuous extension of  $F_\varphi$ to $BV$, less than or equal to $F_\varphi$. 
Besides this abstract definition, for the applications it is of paramount 
importance to
have an integral representation of $\rel{F_\varphi}$.
To this end,
the main difficulty is to find a precise representative for the singular part of the relaxed functional
(i.e., the representative where the measure $Du$ is singular). 
In the vectorial case the singular part of this functional is represented in an abstract way by a generalised surface energy density associated with $f$.

One of the main contributions of this paper
is to find an integral representation for $\rel{F}_\varphi(u,A)$. 
More precisely, in Theorem~\ref{t:limsupG} below we prove that, 
for every 
$u\in  \BVLloc$ 
and for every open set $A\subseteq\Omega$,
it holds that
\begin{equation}
\label{t:generalenuovo1}
\rel{F_\varphi}(u,A)=\int_A\varphi\ d(\bsmall(x,u),Du) \,,
\end{equation}
where $(\bsmall(x,u),Du)$ denotes the pairing measure defined in the recent paper \cite{CD4}. 
This measure extends the concept of pairing measure introduced by Anzellotti in the celebrated paper \cite{Anz} by establishing a pairing theory between 
weakly differentiable vector fields $\bsmall(x)$ and \(BV\) functions.
While the original definition of this measure starts from a  distributional viewpoint, our contribution shows that it can be regarded also in a variational sense as a relaxed functional.
In the authors' opinion,
this variational interpretation 
may attract the attention of the community
that deals with 
%This variational interpretation seems to be useful in order to study 
the $1$--Laplace operator,
both in the case of the associated Euler--Lagrange equations (see \cite{Mazon2016})
and in the study of the related Dirichlet problem with measure data (see \cite{LeonardiComi}).

Lower semicontinuity results and representation formulas for the relaxed functional in $BV(\Omega)$ have been obtained by many authors. 
In the already cited paper \cite{DM}, 
Dal Maso showed that, in order to prove lower semicontinuity,
in his general setting the coercivity assumption cannot be dropped.
In the spirit of the alternatives of Serrin
(see \cite{Serrin61}), in order to drop this assumption, 
Fonseca and Leoni in \cite{FL} assumed a uniform lower semicontinuity condition in $x$. Moreover, in \cite{ADCF1,ADCF2}, 
the authors required weak differentiability in $x$ and $BV$  in $x$ dependence, respectively. In these cases the precise representatives for the singular parts are the approximately continuous representative and the lower semicontinuous capacitary representative, respectively.

We remark that,
in the vectorial case (i.e., when $u$ is $\R^m$-valued, case that we do not handle),
the lower semicontinuity and the relaxation have been studied by many authors
under the quasi-convexity assumption with respect to the gradient variable
(see the fundamental paper \cite{FM} and the recent works \cite{RS,RS1}).
%in the vectorial case, starting from the fundamental paper \cite{FM}, the lower semicontinuity and the relaxation has been studied by many authors until the recent works \cite{RS,RS1}, under the quasi-convexity assumption with respect to the gradient variable. 
In all the mentioned papers, a coercivity (or partial coercivity)
assumption is required. 
It is well known that, in this setting,
the relaxation result can be obtained requiring only the bare
measurability in the space variable $x$.
The coercivity assumption is strictly related to the choice of the weak$^*$ convergence, since under this assumption every sequence that converges in $L^1$ 
converges also in the weak$^*$ topology.
On the other hand, 
although our results are confined to the scalar case,
this improvement by itself (the $L^1$ versus the weak$^*$ relaxation) is crucial in order to apply the results to problems involving very degenerate functionals, like the ones that arise for instance in the weighted least gradient problem. 
Furthermore, due to the lack of coercivity, 
the optimal finiteness domain of the relaxed functional could be larger than $BV$
%it is not possible to fix a priori the finiteness space of the relaxed functional, which could be larger than $BV$ 
(see \cite{CDS}).

\smallskip

Before describing in more details our results,
a few words on the pairing measure are in order.

The pairing theory was initially used to extend the validity of the
Gauss--Green formula to divergence-measure vector fields and to 
non-smooth domains (see \cite{Anz,BuCoMi,Cas,ChenFrid,ChTo,ChToZi,ComiLeonardi,ComiPayne,CD3,LeoSar}). 
%In \cite{CD3}, in the general case of 
%divergence measure vector fields, we give 
%a characterization
%of the normal traces of these vector fields 
%and an analysis of the singular part of the pairing measure.
Moreover, it can be considered 
as a useful abstract tool in several contexts, 
ranging from applications in the theory of hyperbolic systems of conservation and balance laws
(see \cite{ChFr1,ChenFrid,ChTo2,ChTo,ChToZi,CD2,Pan1} and the references therein)
to the theory of capillarity and 
in the study of the Prescribed Mean Curvature problem (see e.g.\ \cite{LeonardiComi,LeoSar,LeoSar2})
and in the context of continuum mechanics 
(see e.g.\ \cite{DGMM,Silh,Schu}).
%to the case of vector fields induced by functions of bounded deformation (see \cite{AmbCriMan,ADM}).
Another field of application is related to
the Dirichlet problem for equations involving the $1$--Laplacian operator (see \cite{ABCM,AVCM,Cas,
%DeGiOlPe,DeGiSe,
HI,SchSch,SchSch2}).
%We also remark that, in some lower semicontinuity problems for integral functionals defined in Sobolev spaces and in $BV$, 
%vector fields with measure--derivative occurred as natural dependence  of the integrand with respect to the spatial variable (see \cite{BouDM,DCFV2,dcl}).

%COLLEGAMENTO CON DAL MASO BOUCHITTE DA SCRIVERE BENE
%In the particular case when $f(x,\xi)$ does not depend on $t$, Bouchitte-DalMaso proved that the relaxed functional $\F_f(u)$ admits a similar integrand in the singular part
%$$
%h_K(x,\nu_u)=\sup_{\bsmall\in D}\Cyl{\bsmall}{\nu_u}{x}
%$$
%with
%$D$ sequentially dense in $K\subseteq \LDMloc$.

In a recent paper \cite{CD4},
the authors introduced a nonlinear version of the pairing suitable for
applications to semicontinuity problems. 
This is the pairing appearing in the representation formula~\eqref{t:generalenuovo1}. The pairing $(\bsmall(x,u),Du)$ generalizes the Anzellotti pairing and inherits its properties. 
In particular, in that paper it is proved a characterization of
the normal traces of the vector field
$\bsmall(x,u(x))$
and the singular part of the pairing measure has been analyzed. 
Moreover, the authors established a generalized Gauss--Green formula.

\smallskip
Let us now describe the contents of the present paper.
We underline that all our new results have been obtained assuming
that the divergence of the vector field with respect to $x$ is an $L^1$ function (see Section~\ref{ss:assumptions} for the detailed list of assumptions on $\bsmall$).
We believe that this can be considered as a first step in order to study the general case with a divergence-measure vector field.

In Section~\ref{ss:coarea},
we prove a coarea formula for the measure $(\bsmall(x,u),Du)$
and its variation
(see Theorems~\ref{p:coarea} and~\ref{p:coareavar}).

Then, in Section~\ref{repres},
we show that the pairing $(\bsmall(x,u),Du)$ admits a representation of the form
%\begin{equation}
%\label{represent1}
%\begin{split}
%\int_\Omega(\bsmall(\cdot,u), Du) = {} &
%\int_\Omega\bsmall(x,u)\cdot\nabla u\,dx+\int_\Omega\Cyl{\bsmall_{\ut}}{\nu_u}{\cdot}\, |D^d u|\\ & 
%+ \int_\Omega\left(\mean{u^-}^{u^+} \Cyl{\bsmall_t}{\nu_u}{\cdot}\, dt\right)
%\, |D^j u|,\qquad u\in \BVLloc
%\end{split}
%\end{equation}
\begin{equation}
\label{represent1}
\begin{split}
(\bsmall(\cdot,u), Du) = {} &
\bsmall(x,u)\cdot\nabla u\,\Leb^N
+ \Cyl{\bsmall_{\ut}}{\nu_u}{\cdot}\, |D^c u|\\ & 
+ \left(\mean{u^-}^{u^+} \Cyl{\bsmall_t}{\nu_u}{\cdot}\, dt\right)
\, |D^j u|,\qquad u\in \BVLloc\,,
\end{split}
\end{equation}
where $\bsmall_t := \bsmall(t, \cdot)$,
the approximate limit $\ut(x)$ of $u$ at $x$
and the normal vector $\nu_u(x)$ are defined 
in the first part of Section~\ref{s:prelim},
and
$\Cyl{\bsmall_t}{\nu_u}{\cdot}$
plays the role of a precise representative, defined by means of some cylindrical averages (see \eqref{defq} below). 
The above formula extends the representation formula for the pairing obtained by Anzellotti in the unpublished paper \cite{Anz2}
in the case of a vector field $\bsmall(x)$ independent of $u$.

In the same spirit, we prove a similar representation formula
\begin{equation}
\label{f:repr2}
\begin{split}
& (\bsmall(\cdot,u), Du)  =  
\bsmall(x,u)\cdot\nabla u\,\Leb^N
+
\Trace[]{\bsmall(\cdot,\widetilde u)}{\partial^* \{u>\widetilde u(x)\} }(x)|D^cu|\\
& +
\left(\mean{u^-(x)}^{u^+(x)}
\Trace[]{\bsmall(\cdot,t)}{\partial^* \{u>t\} }(x)
\, dt\right) |D^ju|,\quad u\in \BVLloc\,,
\end{split}
\end{equation}
based on the use of the weak normal traces as precise representatives
(see Section~\ref{distrtraces} for their definition).
This formula generalizes the representation obtained in the recent paper \cite{CCDM}
for vector fields $\bsmall(x)$ independent of $u$.

\smallskip
Sections~\ref{ss:lower} and~\ref{s:relax}
are devoted to the study of semicontinuity and relaxation.
We start, in Proposition~\ref{p:cont}, by proving that,
under suitable assumptions on $\bsmall$,
the functional
$G_\varphi(u) := \int_\Omega\varphi\ d(\bsmall(x,u),Du)$
is continuous with respect to the $L^1$ and weak${}^*$--BV convergence.
This result is used in Theorem~\ref{p:lsc} to prove the
lower semicontinuity of the pairing functionals
$F(u) := |(\bsmall(\cdot, u), Du)|(\Omega)$
and
$G^+(u) := (\bsmall(\cdot, u), Du)^+(\Omega)$.
Furthermore, the same continuity result
is used in Section~\ref{s:relax}
to prove the representation formula~\eqref{t:generalenuovo1} for the relaxed functional.
%
%for every $u\in \BVLloc$, it holds that
%\begin{equation}
%\label{represent12}
%\begin{split}
%\overline F_\varphi(u,\Omega) = 
%\int_\Omega \varphi \, d (\bsmall(\cdot,u), Du)\,.
%%{} &
%%\int_\Omega\varphi\ \bsmall(\cdot,u)\cdot\nabla u\,dx+\int_\Omega\varphi\ \Cyl{\bsmall_{\ut}}{\nu_u}{\cdot}\, |D^c u|\\ & 
%%+ \int_\Omega\varphi\left(\mean{u^-}^{u^+} \Cyl{\bsmall_t}{\nu_u}{\cdot}\, dt\right)
%%\, |D^j u|.
%\end{split}
%\end{equation}
Clearly, this representation formula, coupled with
\eqref{represent1} or \eqref{f:repr2},
gives a full integral representation of the relaxed functional.
In order to achieve \eqref{t:generalenuovo1}, we need to prove 
those which in relaxation theory are called
the liminf and the limsup inequalities. 
The first one is a consequence of the semicontinuity result proved in Proposition~\ref{p:cont}, 
while the second one is obtained by using the blow-up method
developed in \cite{FL1}.

%%%%%%%%%%%%%%%%%%%%%%%%%%%%%%%%%%%%%%%%%%%
\section{Preliminaries}
\label{s:prelim}

%Throughout the paper, $N\geq2$ is a fixed integer and
%the letter $c$ denotes a strictly positive constant, whose
%value may vary from line to line.
Given $x_0\in\RN$ and $\rho>0$, $B_\rho(x_0)$ denotes the ball in $\RN$
centered in $x_0$ with radius $\rho$, while $\sfera^{N-1}$ is the
unit sphere of $\R^N$.

\noindent

In the following \(\Omega\) will always denote a nonempty open subset of \(\R^N\).
We denote by $\M(\Omega)$ the space of signed Radon measures
on $\Omega$.

\noindent
As usual,  ${\mathcal L}^{N}$ stands for the Lebesgue measure on $\RN$
and ${\H}^{k}$ for the $k$-dimensional Hausdorff
measure  on $\RN$.
The Lebesgue measure of the unit ball in $\RN$ is denoted by $\pallaunit$, hence
$\L^N(B_\rho(x_0))=\pallaunit\rho^N$.

%For every $x\in \R^N$, $\Ixr(y) := (y-x) / r$ denotes the homothety with 
%scaling 
%factor $r$ mapping $x$ in $0$, and the pushforward $\Ixr[x, r_i]_\#\mu$ of 
%a Radon measure $\mu$ in $\R^N$ is the 
%measure  acting on a test function $\phi$ as 
%\begin{equation*}
%%\label{push}
%\int_{\R^N}\phi \,d(\Ixr[x, r]_\#\mu)=\int_{\R^N}\phi\circ \Ixr\, d\mu.
%\end{equation*}

Let \(u\in L^1_{{\rm loc}}(\Omega, \R^m)\).
We say that \(u\) has an approximate limit at $x_{0}\in\Omega$ 
if there exists \(z\in\R^m\) such that
\begin{equation*}
%\label{f:apcont}
	\lim_{r\rightarrow0^{+}}\frac{1}{\LLN\left(  B_r(x_0)\right)}\int_{B_r\left(  x_{0}\right)
	}\left|  u(x)  -z  \right|  \,dx=0.
\end{equation*}
The set $C_u\subset\Omega$ of points where this property holds is called the
\textsl{approximate continuity set} of $u$, whereas
the set \(S_u := \Omega\setminus C_u\)
is called the \textsl{approximate discontinuity set} of $u$.
For any $x\in C_u$ the approximate limit $z$ is uniquely 
determined and is denoted by $z:=\widetilde{u}(x)$.

We say that \(x\in\Omega\) is an {\sl approximate jump point} of \(u\) if
there exist \(a,b\in\R^m\), \(a\neq b\), and a unit vector \(\nu\in\R^N\) such that 
\begin{equation}
\label{f:disc}
\begin{gathered}
\lim_{r \to 0^+} \frac{1}{\LLN(B_r^i(x))}
\int_{B_r^i(x)} |u(y) - a|\, dy = 0,
\\
\lim_{r \to 0^+} \frac{1}{\LLN(B_r^e(x))}
\int_{B_r^e(x)} |u(y) - b|\, dy = 0,
\end{gathered}
\end{equation}
where \(B_r^i(x) := \{y\in B_r(x):\ (y-x)\cdot \nu > 0\}\), and 
\(B_r^e(x) := \{y\in B_r(x):\ (y-x)\cdot \nu < 0\}\).
The triplet \((a,b,\nu)\), uniquely determined by \eqref{f:disc} 
up to a permutation
of \((a,b)\) and a change of sign of \(\nu\),
is denoted by \((u^+(x_0), u^-(x_0), \nu_u(x_0))\).
%\((\uint(x), \uext(x), \nu_u(x))\).
The set of approximate jump points of \(u\) will be denoted by \(J_u\).

The space $BV(\Om)$ is defined as the space of all functions 
$u\colon\Omega\to \R$
belonging to $L^1(\Om)$ whose distributional gradient $Du$ is an
$\RN$-valued Radon measure
%(i.e., $Du\in \M(\Om;\RN)$)
with total variation $\vert Du\vert$ bounded in $\Om$.
We indicate by $D^a u $ and $D^s u$ the {\it absolutely continuous} and
the {\it singular part} of the measure $Du$ with respect to
the Lebesgue measure. We recall that $D^a u$ and $D^s u$ are mutually
singular, moreover we can write
$$
Du = D^a u + D^s u \qquad {\rm and} \qquad D^a u = \nabla u ~\L^N,
$$
where $\nabla u$ is the {\it Radon-Nikod\'ym derivative} of $D^au$ with respect to
the Lebesgue measure.
We denote by $Du = \nu_u\, |Du|$ the polar decomposition of $Du$.
In addition, it holds that
$$
D^su =  D^c u + D^j u,
\quad
\text{with}\
D^j u = (u^+-u^-)\nu_u~\H^{N-1}\ristretto{J_u}\,,
$$
where $J_u$ is a {\it countably $\H^{N-1}$-rectifiable} Borel set (see
\cite[Definition 2.57]{AFP}) contained in $S_u$, such that $\H^{N-1}(S_u\setminus J_u)=0$.
The remaining part $D^c u$ is called the {\it Cantor part} of $Du$.
We will also use the notation $D^d u:=D^a u + D^c u$ for the diffuse part of the measure $Du$.

A set $E\subset\Omega$ is of finite perimeter if 
its characteristic function $\chi_E$ belongs to $BV(\Omega)$.
If \(\Omega\subset\R^N\) is the largest open set such that \(E\)
is locally of finite perimeter in \(\Omega\),
we call \textsl{reduced boundary} \(\partial^* E\) of \(E\) the set of all 
points
\(x\in \Omega\) in the support of \(|D\chi_E|\) such that the limit
\[
\nuint_E(x) := \lim_{\rho\to 0^+} 
\frac{D\chi_E(B_\rho(x))}{|D\chi_E|(B_\rho(x))}
\]
exists in \(\R^N\) and satisfies \(|\nuint_E(x)| = 1\).
The function \(\nuint_E\colon\partial^* E\to \R^N\) is called
the \textsl{measure theoretic unit interior normal} to \(E\).

A fundamental result of De Giorgi (see \cite[Theorem~3.59]{AFP}) states that
\(\partial^* E\) is countably \((N-1)\)-rectifiable,
\(|D\chi_E| = \hh\res \partial^* E\),
and $\nuint_E$ coincides (up to the sign) with the normal
$\nu_{\partial^* E}$ defined in Section~\ref{distrtraces}.
Moreover, the measure theoretic interior normal can be choosen as normal vector
to $\partial^* E$, in the sense of Section \ref{distrtraces}.

If $u\in BV(\Omega)$, then the level set
$E_t := \{u > t\}$ is of finite perimeter for a.e.\ $t\in\R$,
and we can choose the sign of the normal vectors so that
$\nuint_{E_t}(x) = \nu_{\Sigma_{t}}(x) = \nu_u(x)$
for $\hh$-a.e.\ $x\in\Sigma_t$,
where $\Sigma_t := \partial^*\{u>t\}$.

Moreover, we can choose an orientation on $J_u$ such that
$u^+(x) > u^-(x)$ for every $x\in J_u$.
We also set $u^-(x)=u^+(x) := \ut(x)$ for every $x\in C_u$,
and $u^*(x) := [u^+(x) + u^-(x)]/2$  for every $x\in C_u \cup J_u$.

The measure $Du$ can be disintegrated on the level sets of $u$
using the following coarea formula
(see \cite[Theorem 4.5.9]{Fed}).

\begin{theorem}[Coarea formula]\label{coarea}
If $u\in BV(\Omega)$, then for $\LLU$--a.e.\ $t\in\R$ the set $\{u>t\}$ has finite perimeter in $\Omega$ and  
the following coarea formula holds
\begin{equation*}
\int_\Omega
g\,d|Du|=\int_{-\infty}^{+\infty}\!\!dt\!\int_{\partial^* \{u>t\}
\cap\Omega}\!g\,d\hh
=
\int_{-\infty}^{+\infty}\!\!dt\!
\int_{\{u^-\leq t\leq u^+ \}}\!g\,d\hh\,,
\end{equation*}
for every Borel function $g:\Omega\to[0,+\infty]$.
Moreover, for $\LLU$--a.e. $t\in\R$,
\begin{itemize}
\item[(a)]
 $\partial^*{\{u>t\}} \subset \{u^-\leq t\leq u^+ \}$,
\item[(b)] $\hh\Bigl(\{u^-\leq t\leq u^+ \}\setminus
\bigl(\partial^*\{u>t\}\bigr)\Bigr)=0
$,
\end{itemize}
and, in particular,
\begin{itemize}
\item[(a$^\prime$)]
$
\partial^* \{u > t\} \cap (\Omega\setminus S_u)
\subseteq \{x\in \Omega\setminus S_u\colon \ut(x) = t\}
$,
\item[(b$^\prime$)]
$\hh\Bigl(\{x\in \Omega\setminus S_u\colon \ut(x)=t \}\setminus
\bigl((\Omega\setminus S_u)\cap\partial^*\{u>t\}\bigr)\Bigr)=0$.
\end{itemize}
\end{theorem}

%%%%%%%%%%%%%%%%%%%%%%%%%%%%%%%%%%%%%%%%%%%%%%%%%%%%%%%%%%%%%%%%
\subsection{Divergence-measure fields }
\label{ss:div}

We will denote by \(\DM(\Omega)\) the space of all
vector fields 
\(\A\in L^\infty(\Omega, \R^N)\)
whose divergence in the sense of distributions is a bounded Radon measure in \(\Omega\).
Similarly, \(\DMloc[\Omega]\) will denote the space of
all vector fields \(\A\in L^\infty_{{\rm loc}}(\Omega, \R^N)\)
whose divergence in the sense of distribution is a Radon measure in \(\Omega\). 
We set \(\DM = \DM(\R^N)\).
Moreover, we denote by $\LDM$ (resp.\ $\LDMloc$)
the subset of $\DM(\Omega)$ (resp.\ $\DMloc$)
of vector fields whose divergence is in $L^1(\Omega)$
(resp.\ $L^1_{\text{loc}}(\Omega)$).

We recall that, if \(\A\in\DMloc[\Omega]\), then \(|\Div\A| \ll \hh\)
(see \cite[Proposition 3.1]{ChenFrid}).
As a consequence, the set
\begin{equation*}
%\label{f:jump}
	\jump{\A} 
	:= \left\{
	x\in\Omega:\
	\limsup_{r \to 0+}
	\frac{|\Div \A| (B_r(x))}{r^{N-1}} > 0
	\right\},
\end{equation*} 
is a Borel set, \(\sigma\)-finite with respect to \(\hh\),
and the measure \(\Div \A\) can be decomposed as
\[
\Div\A = \Div^a\A + \Div^c\A + \Div^j\A,
\]
where \(\Div^a\A\) is absolutely continuous with respect to \(\LLN\),
\(\Div^c\A (B) = 0\) for every set \(B\) with \(\hh(B) < +\infty\),
and
\[
\Div^j\A = h\, \hh\res\jump{\A}
\]
for some Borel function \(h\)
(see \cite[Proposition~2.3]{ADM}).

\subsection{Anzellotti's pairing}
As in Anzellotti \cite{Anz} (see also \cite{ChenFrid}),
for every \(\A\in\DMloc\) and \(u\in\BVLloc\) we define the linear functional
\((\A, Du) \colon C^\infty_0(\Omega) \to \R\) by
\begin{equation*}
%\label{f:pairing}
	\pscal{(\A, Du)}{\varphi} :=
	-\int_\Omega u^*\varphi\, d \Div \A - \int_\Omega u \, \A\cdot \nabla\varphi\, dx. 
\end{equation*}
The distribution \((\A, Du)\) is a Radon measure in \(\Omega\),
absolutely continuous with respect to \(|Du|\)
(see \cite[Theorem 1.5]{Anz} and \cite[Theorem 3.2]{ChenFrid}),
hence the equation
\begin{equation*}
%\label{f:anz}
	\Div(u\A) = u^* \Div\A + (\A, Du)
\end{equation*}
holds in the sense of measures in \(\Omega\).
%(We remark that, in \cite{ChenFrid}, the measure \((\A, Du)\) is denoted
%by \(\overline{\A\cdot Du}\).)
Furthermore, Chen and Frid in \cite{ChenFrid} proved that the absolutely continuous part
of this measure with respect to the Lebesgue measure is given by
\(
(\A, Du)^a = \A \cdot \nabla u\, \LLN
\).

In \cite{Anz2} it is proved that,
for every \(\A\in \LDMloc\),
\[
(\A, Du) = \Cyl{\A}{\nu_u}{\cdot}\, |Du|\,,
\qquad \text{$|Du|$--a.e.\ in $\Omega$,}
\]
where $Du = \nu_u\, |Du|$ and
\begin{equation}\label{defq}
\Cyl{\A}{\nu_u}{x}
:= \lim_{\rho\downarrow 0}
\lim_{r\downarrow 0}
\frac{1}{\LLN(C_{r,\rho}(x, \nu_u(x)))}
\int_{C_{r,\rho}(x, \nu_u(x))} \A(y) \cdot \nu_u(x)\, dy
\end{equation}
and, for every $\zeta\in\sfera^{N-1}$,
\begin{equation*}
%\label{defcilindro}
C_{r,\rho}(x, \zeta) :=
\left\{
y\in\R^N:\ 
|(y-x)\cdot\zeta| < r,\
|(y-x) - [(y-x)\cdot\zeta]\zeta| < \rho
\right\}.
\end{equation*}
(The existence of the limit in \eqref{defq} for $|Du|$--a.e.\ $x\in\Omega$ is part of the statement.)

As a consequence, it holds that
\[
\lim_{r\downarrow 0} \frac{(\A, Du) (B_r(x))}{|Du|(B_r(x))}
= \Cyl{\A}{\nu_u}{x}
\qquad
\text{for $|Du|$--a.e.\ $x\in\Omega$}.
\]
The interested reader can find the results of the unpublished paper \cite{Anz2},
proved in a more general setting,
in \cite[Section~4]{CCDM}.

\subsection{Weak normal traces on oriented countably $\Haus{N-1}$-rectifiable sets} \label{distrtraces}
We recall the notion of the traces of the normal component of a vector field \(\A\in 
\DMloc\) on an oriented countably \(\Haus{N-1}\)--rectifiable set
\(\Sigma\subset\Omega\),
introduced in 
\cite[Propositions~3.2, 3.4 and Definition~3.3]{AmbCriMan}.
In that paper the authors proved that there exist the normal traces $\Trp{\A}{\Sigma}, \Trm{\A}{\Sigma}$ belonging to
\(L^{\infty}(\Sigma, \Haus{N-1}\res\Sigma)\) 
and satisfying
\begin{equation}\label{f:trA}
\Div \A \res\Sigma =
\left[\Trp{\A}{\Sigma} - \Trm{\A}{\Sigma}\right]
\, {\mathcal H}^{N-1} \res\Sigma.
\end{equation}
%In particular, 
%$|\Div\A|(\Sigma) \leq 2\|\A\|_\infty \Haus{N-1}(\Sigma)$.

In what follows we use the notation
\[
\Trace[*]{\A}{\Sigma}:= \frac{\Trp{\A}{\Sigma}+\Trm{\A}{\Sigma}}{2}\,.
\]
If $\A\in\LDMloc$, then for $\Haus{N-1}$-a.e. $x\in \Sigma$ we have
$
\Trp{\A}{\Sigma}(x)=\Trm{\A}{\Sigma}(x)
$
and
$
\Div \A \res\Sigma =0.
$
Therefore, in the following, when $\A\in\LDMloc$, we will use the notation $\Trace[]{\A}{\Sigma}$, instead of $\Trace[*]{\A}{\Sigma}$.

\subsection{Representation formulas for the pairing measure}
In the following theorem, the pairing is characterized in terms of normal 
traces of the field $\A$ on level sets of $u$.

\begin{theorem}[see \cite{CCDM}, Thm.~3.9]
\label{t:repr}
Let $\A \in \DMloc[\R^N]$ and $u\in\BVLloc[\R^N]$.
Then, the following equality holds in the sense of measures
\begin{equation}
\label{f:repr}
(\A, Du) =\left(
\mean{u^-(x)}^{u^+(x)}
\Trace[*]{\A}{\partial^* \{u>t\} }(x)
\, dt\right)|Du|\,,
\end{equation}
where we use the convention $\mean{a}^a f(t)\, dt := f(a)$.
Moreover, 
\begin{itemize}
	\item[(i)]
	absolutely continuous part: 
	\((\A, Du)^a = \A \cdot \nabla u\, \LLN\);
	%\(\LLN\)-almost everywhere in \(\Omega\);
	
	\item[(ii)]
	jump part:
	$
(\A,Du)^j = \Trace[*]{\A}{J_u}(x)|D^j u|;
$

	\item[(iii)]
	Cantor part:
$
(\A, Du)^c =
\Trace[*]{\A}{\partial^* \{u>\widetilde u(x)\} }(x)|D^cu|.
$
\end{itemize}

\end{theorem}

\begin{corollary}\label{c:paironch}
Let \(\A\in\DMloc\), and let $E\subseteq \Omega$ be a set of finite perimeter,
with $\overline{E}\subset\Omega$.
Then
\[
(\A,D\chi_E)=\Trace[*]{\A}{\partial^* E}(x)\hh\res\partial^* E.
\]
\end{corollary}

\section{Assumptions on  the vector field \(\bsmall\)}
%\(\B\)}
\label{ss:assumptions}

Let \(\bsmall\colon \Omega\times \R \to \R^N\) be a function
satisfying the following assumptions:
\begin{itemize}
\item[(i)]
\(\bsmall\) is a locally bounded Borel function;

%\item[(v)]
\item[(ii)]
the function \(\bsmall(x, \cdot)\) is Lipschitz continuous in \(\R\),
uniformly with respect to $x$, i.e.\
there exists a constant $L>0$ such that
\[
|\bsmall(x,t) - \bsmall(x, s)| \leq L\, |t-s|,
\qquad
\forall t,s\in\R,\
\text{for \(\LLN\)--a.e.\ \(x\in\Omega\)}\,;
\]

\item[(iii)]
for every \(t\in\R\), \(\bsmall_t := \bsmall(\cdot, t)\in\LDMloc\);

\item[(iv)]
the least upper bound
\begin{equation}
\label{f:sigma}
\sigma := \sup_{t\in \R} |\Div_x \bsmall_t|
\end{equation}
belongs to $L^1_{\rm loc}(\Omega)$.
\end{itemize}

\smallskip
We remark that, at the price of some additional
technicality, 
%when dealing with functions $u\in \BVLloc$, 
assumption (iv) could be replaced by the weaker assumption
\begin{itemize}
\item[(iv$^\prime$)]
for every $m>0$,
the least upper bound
\[
\sigma_m := \sup_{|t|\leq m} |\Div_x \bsmall_t|
\]
belongs to $L^1_{\rm loc}(\Omega)$.

\end{itemize}

The results of Section~\ref{ss:coarea} 
will be mainly proved replacing (ii) with the following weaker assumption:
\begin{itemize}
\item[(ii$^\prime$)]
for \(\LLN\)--a.e.\ \(x\in\Omega\),
the function \(\bsmall(x, \cdot)\) is continuous in \(\R\).
\end{itemize}

\bigskip

Let us extend $\bsmall$ to be $0$ in $(\R^N\setminus\Omega)\times\R$,
so that the vector field
\begin{equation}\label{f:B}
\B(x,t) := \int_0^t \bsmall(x,s)\, ds,
\qquad x\in\R^N,\ t\in\R,
\end{equation}
is defined for all $(x,t)\in\R^N\times\R$.
Moreover \(\B(x,0) = 0\) for every \(x\in\R^N\)
and, from (ii$^\prime$), for every \(x\in\R^n\)
one has \(\bsmall(x,t) = \partial_t \B(x,t)\) for every \(t\in\R\).

The next theorem has been proved,
in a more general setting,
in \cite[Theorem~4.3]{CD4}.
To help the reader comparing \cite[formula (29)]{CD4}
with formula~\eqref{lolo} below,
we observe that
the quantity $F(x,t)\sigma$ in \cite{CD4}
coincides with
$(\Div_x\B)(x,t)\, \LLN$ 
in the present paper.

\begin{theorem}%[See \cite{CD4}]
\label{t:chainb4}
Let $\bsmall$
satisfy assumptions (i)-(ii$^\prime$)-(iii)-(iv),
let \(\B\) be defined by~\eqref{f:B},
and let $u\in  \BVLloc$. 
Then the distribution $(\bsmall(\cdot,u), Du)$,
defined by
\begin{equation}
\label{f:bxu}
\begin{split}
 \langle(\bsmall(\cdot,u), Du),\varphi\rangle  := {} &
-\int_\Omega
\varphi(x)\, (\Div_x \B) (x,u(x)) \, dx
\\ & -\int_\Omega \B(x,u(x))\cdot\nabla\varphi(x)\, dx,
\qquad \forall\varphi\in C_c^\infty(\Omega),
\end{split}
\end{equation}
is a Radon measure in $\Omega$, and satisfies
\begin{equation}\label{f:mubdd}
|(\bsmall(\cdot,u), Du)|(E)\leq
\|\bsmall\|_{L^\infty(K, \R^N)}
|Du|(E),
\qquad
\text{for every Borel set}\
E\Subset\Omega\,,
\end{equation}
where $K := \overline{E} \times \left[-\|u\|_{L^\infty(\overline{E})}, \|u\|_{L^\infty(\overline{E})}\right]$.

In other words,
the composite function $\vv\colon\Omega\rightarrow\R^N$,
defined by
$\vv(x):=\B(x,u(x))$,
belongs to $L^\infty_{\rm loc}(\Omega, \R^N)$,
and the following equality holds in the sense of measures:
\begin{equation}\label{lolo}
\begin{split}
\Div\,\vv =
(\Div_x\B)(x,u(x))\, \LLN
+(\bsmall(\cdot,u), Du).
\end{split}
\end{equation}
\end{theorem}

From \eqref{f:mubdd} it follows that
\((\bsmall(\cdot,u), Du) \ll |Du| \),
hence there exists a function
\(\Theta(\bsmall, u; \cdot) \in L^1(\Omega, |Du|)\)
such that
\begin{equation}\label{f:Theta}
(\bsmall(\cdot,u), Du) = \Theta(\bsmall, u; \cdot)\, |Du|\,,
\qquad \text{$|Du|$--a.e.\ in $\Omega$.}
\end{equation}

\begin{remark}
\label{r:repr}
By the definition \eqref{f:bxu} of the pairing and the definition \eqref{f:B} of $\B$, it follows that,
for every $\varphi\in C_c^\infty(\Omega)$,
\begin{equation}
\label{f:bxurepr}
\begin{split}
 \langle(\bsmall(\cdot,u), Du),\varphi\rangle  = {} &
-\int_\Omega \varphi(x) \int_0^{u(x)} \Div_x \bsmall_t(x)\, dt\, dx
\\ & - \int_\Omega \int_0^{u(x)} \bsmall_t(x)\cdot \nabla\varphi(x)\, dt\, dx.
\end{split}
\end{equation}
\end{remark}

\section{Coarea formula for the pairing measure}
\label{ss:coarea}

In this section we establish a coarea formula for the pairing measure
$(\bsmall(\cdot, u), Du)$,
and we draw some consequences
that will be used in order to prove its integral representation (see Theorem \ref{t:repr2} below).

\begin{theorem}[Coarea formula for the pairing measure]
\label{p:coarea}
Let $\bsmall$
satisfy assumptions (i)-(ii$^\prime$)-(iii)-(iv),
and let $u\in  \BVLloc$. 
Then
\begin{gather}
 \langle(\bsmall(\cdot,u), Du),\varphi\rangle
= \int_{\R}  \langle(\bsmall_t, D\chi_{\{u>t\}}),\varphi\rangle\,dt,
\qquad \forall\varphi\in C_c^\infty(\Omega)\,,
\label{f:coarea}
\\
(\bsmall(\cdot,u), Du) (B)
= \int_{\R}  (\bsmall_t, D\chi_{\{u>t\}}) (B)\,dt,
\qquad \forall\ \text{Borel set}\ B\subset\Omega\,.
\label{f:coareaB}
\end{gather}
\end{theorem}

\begin{proof}
Assume, for simplicity, that $u\geq 0$ and let $C> \|u\|_\infty$.
%By the definition \eqref{f:bxu} we have that
Using the representation \eqref{f:bxurepr}, we have that
\begin{equation*}
%\label{f:co1}
\begin{split}
 \langle(\bsmall(\cdot,u), Du),\varphi\rangle  = {} &
-\int_0^C \int_\Omega \chi_{\{u > t\}} \varphi \Div_x \bsmall_t \, dx\, dt
-\int_0^C \int_\Omega  \chi_{\{u > t\}} \bsmall_t\cdot\nabla\varphi\, dx\, dt
\\ = {} &
\int_0^C \langle(\bsmall_t, D\chi_{\{u>t\}}),\varphi\rangle\,dt,
\end{split}
\end{equation*}
where, in the last equality, we have used the fact that,
for $\LLU$-a.e.\ $t\in\R$,
\[
\Div (\chi_{\{u>t\}} \bsmall_t) = 
\chi_{\{u>t\}}^* \, \Div \bsmall_t + (\bsmall_t, D \chi_{\{u>t\}}).
\]
The general case follows
as in \cite[Theorem~4.2]{CD3} with minor modifications.

Finally, since both sides of \eqref{f:coarea} are real measures in $\Omega$,
they coincide not only as distributions, but also as measures,
hence \eqref{f:coareaB} follows.
\end{proof}

The following approximation result is in the spirit of
\cite[Proposition~4.11]{CDM}, \cite[Proposition~4.15]{CD3},
\cite[Theorem~1.2]{ChenFrid}, \cite[Lemma~2.2]{Anz}.
\begin{theorem}[Approximation by $C^\infty$ fields]
\label{t:convolu}
Let $\bsmall$
satisfy assumptions (i)-(ii$^\prime$)-(iii)-(iv).
Then there exists a sequence of vector fields $\bsmall^k\colon\Omega\times\R\to\R^N$
satisfying the same assumptions, such that
$\bsmall^k_t \in C^\infty(\Omega, \R^N)$ for every $t\in\R$ and 
\[
(\bsmall^k(\cdot, u), Du) \stackrel{*}{\rightharpoonup}
(\bsmall(\cdot, u), Du),
\qquad \forall u\in\BVLloc,
\]
locally in the weak${}^*$ sense of measures in $\Omega$.
If, in addition, $\bsmall$ satisfies (ii),
then also the vector fields $\bsmall^k$ satisfy (ii).
\end{theorem}

\begin{proof}
Using the same construction described in the proof of
\cite[Proposition~4.11]{CDM},
we obtain locally uniformly bounded vector fields 
$\bsmall^k_t \in C^\infty(\Omega, \R^N)$ satisfying
(i)-(ii$^\prime$)-(iii)-(iv),  and, for every $t\in\R$,
$\bsmall^k_t \to \bsmall_t$ in $L^1_{\rm loc}(\Omega)$.
If, in addition, $\bsmall$ satisfies (ii),
then it is verified that also the vector fields $\bsmall^k$ satisfy (ii).

Moreover, for every $t\in\R$ and $v\in\BVLloc$,
\[
\lim_{k\to +\infty}\int_{\Omega} v \, \varphi \, \Div\bsmall^k_t\, dx
= \int_{\Omega} v^* \,\varphi \, d \Div\bsmall_t,
\qquad \forall \varphi\in C_c(\Omega)
\]
(see \cite{CDM}, formula (4.8)).
We underline that, since by assumption $\Div\bsmall_t\in L^1_{\rm loc}(\Omega)$,
then the above relation can be written as
\begin{equation}
\label{f:ivp0}
\lim_{k\to +\infty}\int_{\Omega} v \, \varphi \, \Div\bsmall^k_t\, dx
= \int_{\Omega} v \,\varphi \, \Div\bsmall_t\, dx,
\qquad \forall \varphi\in C_c(\Omega)\,.
\end{equation}
Let us fix $u\in \BVLloc$ and $\varphi\in C_c(\Omega)$.
To simplify the notation, we assume without loss of generality
that $u\geq 0$.
By the representation formula \eqref{f:bxurepr} and Fubini's Theorem, we have that
\begin{equation*}
%\label{f:bxuk}
\begin{split}
 \langle(\bsmall^k(\cdot,u), Du),\varphi\rangle  = {} &
-\int_\Omega \varphi(x) \int_0^{u(x)} \Div_x \bsmall^k_t(x)\, dt\, dx
\\ & - \int_\Omega \int_0^{u(x)} \bsmall^k_t(x)\cdot \nabla\varphi(x)\, dt\, dx
\\ = {} &
- \int_0^{\infty} \int_\Omega \chi_{\{u > t\}}(x)\, \varphi(x)\,\Div_x \bsmall^k_t(x)\,dx\, dt
\\ & 
- \int_0^{\infty} \int_\Omega \chi_{\{u > t\}}(x)\, \bsmall^k_t(x)\cdot
\nabla\varphi(x)\, dx \, dt
\\ =: {} & -I^k_1 - I^k_2.
\end{split}
\end{equation*}
For every $t\in\R$, by \eqref{f:ivp0} with $v = \chi_{\{u > t\}}$ we 
deduce that, as $k\to\infty$,
\[
\zeta^k(t) := 
\int_\Omega \chi_{\{u > t\}}(x)\, \varphi(x)\,\Div_x \bsmall^k_t(x)\,dx
\to
\int_\Omega \chi_{\{u > t\}}(x)\, \varphi(x)\,\Div_x \bsmall_t(x)\,dx
\,.
\]
Let $K\Subset \Omega$ denote the support of $\varphi$
and let $a := \|u\|_{L^\infty(K)}$.
Let $\varepsilon > 0$ be such that $K_\varepsilon := K + B_\varepsilon \Subset \Omega$.
By the explicit construction of $\bsmall^k_t$, for $k$ large enough
it holds that
\[
\int_K |\Div_x\bsmall^k_t(x)|\, dx \leq
\int_{K_\varepsilon} |\Div_x\bsmall_t(x)|\, dx + 1
\leq
\int_{K_\varepsilon} \sigma\, dx + 1
\]
(see \cite[formula (4.7)]{CDM}),
where the function $\sigma$ is defined in~\eqref{f:sigma}.
Hence, it holds that
\[
|\zeta^k(t)| \leq \chi_{[0,a]}(t) \, \|\varphi\|_\infty \left(\int_{K_\varepsilon} \sigma\, dx + 1\right)\,.
\]
By the Dominated Convergence Theorem we deduce that
\begin{equation}
\label{f:Ik1}
\lim_{k\to \infty} I^k_1 =
\int_0^{\infty} \int_\Omega \chi_{\{u > t\}}(x)\, \varphi(x)\,\Div_x \bsmall_t(x)\,dx\, dt
= 
\int_\Omega \varphi(x) \int_0^{u(x)} \Div_x \bsmall_t(x)\, dt\, dx\,.
\end{equation}

Let us compute the limit of $I^k_2$.
Since, for every $t\in\R$, $\bsmall^k_t\to \bsmall_t$ in $L^1_{\rm loc}(\Omega)$,
it holds that
\[
\psi^k(t) := \int_\Omega \chi_{\{u > t\}}(x)\, \bsmall^k_t(x)\cdot
\nabla\varphi(x)\, dx \to
\int_\Omega \chi_{\{u > t\}}(x)\, \bsmall_t(x)\cdot
\nabla\varphi(x)\, dx\,.
\]
Moreover, 
%denoting 
%by $K\Subset \Omega$ the support of $\varphi$ and 
%by $a := \|u\|_{L^\infty(K)}$, 
there exists a constant $M> 0$ such that
$\|\bsmall^k\|_{L^\infty(K\times [-a, a])}\leq M$ for every $k\in\N$,
so that
\[
\psi^k(t)| \leq M\, \|\nabla\varphi\|_{1}\,,
\]
and hence, by the Dominated Convergence Theorem,
\begin{equation}
\label{f:Ik2}
\lim_{k\to +\infty} I^k_2
=
 \int_0^{\infty} \int_\Omega \chi_{\{u > t\}}(x)\, \bsmall_s(x)\cdot
\nabla\varphi(x)\, dx \, dt
=
\int_\Omega \int_0^{u(x)} \bsmall_t(x)\cdot \nabla\varphi(x)\, dt\, dx
\,.
\end{equation}
The conclusion now follows from \eqref{f:Ik1} and \eqref{f:Ik2}.
\end{proof}

\begin{proposition}\label{l:theta}
Let $\bsmall$
satisfy assumptions (i)-(ii$^\prime$)-(iii)-(iv),
and let $u\in \BVLloc[\Omega]$.
Then
\begin{equation*}
%\label{f:gthetaut}
\text{for $\mathcal{L}^1$-a.e.}\ t\in\R:
\quad
\Theta(\bsmall, u; x) =
\Theta(\bsmall_t, \chiut; x)
\quad
\text{for \(|D\chiut|\)-a.e.}\ x\in\Omega\,,
\end{equation*}
where $\Theta$ is defined in \eqref{f:Theta}.
\end{proposition}

\begin{proof}
The proof is essentially the same of Proposition~5.2 in
\cite{CDM}, and it is based on the use
of the coarea formula (Theorem~\ref{p:coarea}) and the approximation result by smooth fields (Theorem~\ref{t:convolu}).
\end{proof}

\begin{theorem}[Coarea formula for the variation]
\label{p:coareavar}
Let $\bsmall$
satisfy assumptions (i)-(ii$^\prime$)-(iii)-(iv),
and let $u\in \BVLloc[\Omega]$.
Then
\begin{equation*}
%\label{f:coareavar}
 \langle\left|(\bsmall(\cdot,u), Du)\right|\,,\,\varphi\rangle
= \int_{\R}  \langle\left|(\bsmall_t, D\chi_{\{u>t\}})\right|\,,\,\varphi\rangle\,dt,
\qquad \forall\varphi\in C_c^\infty(\Omega).
\end{equation*}
\end{theorem}

\begin{proof}
To simplify the notation let $\mu := (\bsmall(\cdot,u), Du)$ and
$\mu_t := (\bsmall_t, D\chi_{\{u>t\}})$, $t\in\R$.
By \eqref{f:Theta}, we have that
$\mu = \Theta(\bsmall, u) \, |Du|$, 
with $\Theta(\bsmall, u) \equiv \Theta(\bsmall, u; \cdot)$,
so that 
\[
|\mu| = |\Theta(\bsmall, u)| \, |Du|\,,
\qquad 
|\mu_t| = |\Theta(\bsmall_t, \chiut)|\,  |D\chiut|
\]
(see \cite[Proposition~1.23]{AFP}).
Let $B\subset\Omega$ be a Borel set.
By the coarea formula in BV (see \cite[Theorem~3.40]{AFP}) and
Proposition~\ref{l:theta}
it holds that
\[
\begin{split}
|\mu|(B) & = \int_B |\Theta(\bsmall, u)| \, d|Du|
= \int_{\R} dt \int_B |\Theta(\bsmall, u)|\, d |D\chiut|
\\ &
= \int_{\R} dt \int_B |\Theta(\bsmall_t, \chiut)|\, d |D\chiut|
= \int_{\R} |\mu_t|(B)\, dt\,,
\end{split}
\]
concluding the proof.
\end{proof}

\begin{lemma}
\label{l:lip}
Let $\bsmall$ satisfy (i)--(iv), and let
$u\in  \BVLloc$.
Then, for every $\tau\in\R$
and every $\varphi\in C_c^\infty(\Omega)$, it holds that
\begin{equation*}
%\label{f:lip}
\begin{split}
& \left|\langle(\bsmall(\cdot,u), Du),\varphi\rangle
- \langle(\bsmall_\tau, Du),\varphi\rangle\right|
\\ & \leq
L \, \|\varphi\|_\infty \left[
\int_{\spt\varphi} |\ut-\tau|\, d |D^d u|
+ \int_{J_u\cap\spt\varphi} \left(\int_{u^-}^{u^+} |t-\tau|\, dt\right) d \Haus{N-1}
\right],
%\qquad \forall\varphi\in C_c^\infty(\Omega).
\end{split}
\end{equation*}
where $\spt\varphi \Subset \Omega$ denotes the support of $\varphi$
and $L$ is the Lipschitz constant of assumption~(ii).
\end{lemma}

\begin{proof}
Using the coarea formula \eqref{f:coarea} and (ii)
we obtain that
\begin{equation*}
\begin{split}
I_\tau & :=  \left|\langle(\bsmall(\cdot,u), Du),\varphi\rangle
- \langle(\bsmall_\tau, Du),\varphi\rangle\right|
= \left| \int_\R  \langle(\bsmall_t - \bsmall_\tau, D\chi_{\{u>t\}}),\varphi\rangle\,dt \right|
\\ & \leq
\|\varphi\|_\infty
\int_\R \int_{\spt\varphi} \|\bsmall_t - \bsmall_\tau\|_\infty\, d |D\chi_{\{u>t\}}|\, dt
\\ & \leq
L\, \|\varphi\|_\infty
\int_\R \int_{\spt\varphi} |t - \tau| \, d |D\chi_{\{u>t\}}|\, dt\,.
\end{split}
\end{equation*}
We now consider $\spt\varphi$ as the disjoint union of
$\spt\varphi\setminus J_u$ and  $J_u\cap {\spt\varphi}$,
and we use the coarea formula in $BV$, obtaining
\begin{equation*}
\begin{split}
I_\tau & \leq
L\, \|\varphi\|_\infty
\left[\int_\R \int_{{\spt\varphi}\setminus J_u} |t - \tau| \, d |D\chi_{\{u>t\}}|\, dt
+ \int_\R \int_{J_u\cap {\spt\varphi}} |t - \tau| \, d |D\chi_{\{u>t\}}|\, dt\right]
\\ & =
L \, \|\varphi\|_\infty \left[\int_{{\spt\varphi}} |\ut-\tau|\, d |D^d u|
+ \int_{J_u\cap {\spt\varphi}} \left(\int_{u^-}^{u^+} |t-\tau|\, dt\right) d \Haus{N-1}
\right]\,,
%\qedhere
\end{split}
\end{equation*}
where we used the fact that
$\Haus{N-1}(\Omega\setminus(C_u \cup J_u)) = 0$ so that also
$|D^d u|(\Omega\setminus(C_u \cup J_u)) = 0$.
\end{proof}

\section{Integral representation of the pairing}
\label{repres}
%\todo{[Aggiungere corollario per la variazione totale]}

In this section we are interested in finding an integral representation of the pairing measure and of its total variation. 
We prove that the pairing measure can be represented by an integral functional defined on the space $BV(\Omega)$,
provided that in the support of the singular part of the measure
we choose a suitable precise representative of the vector field $\bsmall$. 

We recall the general form of an integral functional defined in $BV(\Omega)$.
Given the integrand $f(x,t,\xi)=\bsmall(x,t)\cdot\xi$,
for every open set $A\subset\Omega$, let us define the functional 
$\F(\cdot,A)\colon BV(\Om)\to]-\infty,+\infty]$ by setting
\begin{equation*}
%\label{e:fstorto}
\begin{split}
\F(u,A) = {} & \int_A\bsmall(x,u)\cdot \nabla u \, dx 
\\ & +
\int_A\overline f\Bigl(x,\preciso{u},\frac{D^c u}{|D^c u|}\Bigr)\, d|D^c u|
+ \int_{J_u\cap A}d\H^{N-1}\int_{u^-}^{u^+}\overline f(x,t,\nu_u)\,dt,
\end{split}
\end{equation*}
where $\overline f(\cdot,s,\xi)$ is a proper precise representative of $f(\cdot,t,\xi)=\bsmall_t\cdot\xi$,
$\ut(x)$ denotes the approximate limit of $u$ at $x$, and
$(u^+(x), u^-(x), \nu_u(x))$ are defined in Section~\ref{s:relax}.
%We shall often drop the subscript $f$ (if no confusion arises) and write $\F(u)$ in place of $\F_f(u,\Om)$.
\medskip

We show that, in our case, this representative is 
the limit of cylindrical averages introduced in \cite{Anz2} for vector fields $\bsmall(x)$ whose divergence belongs to $L^1$.

%Here we need to adapt this notion of averages to vector field $\bsmall(x,t)$ of this type depending also by $t$.

\begin{theorem}[Integral representation of the pairing measure]
\label{t:repr2}
Let $\bsmall$
satisfy assumptions (i)--(iv),
and let $u\in  \BVLloc$.
Then it holds that
\begin{equation*}
%\label{represent}
\begin{split}
(\bsmall(\cdot,u), Du) = {} &
\Cyl{\bsmall_{\ut}}{\nu_u}{\cdot}\, |D^d u|
%\\ & 
+ \left(\mean{u^-}^{u^+} \Cyl{\bsmall_t}{\nu_u}{\cdot}\, dt\right)
\, |D^j u|.
\end{split}
\end{equation*}
In other words,
the density \(\Theta\) defined in~\eqref{f:Theta} is given by
\begin{equation*}
%\label{represent2}
\Theta(\bsmall, u; x) = 
\begin{cases}
\Cyl{\bsmall_{\ut(x)}}{\nu_u}{x}\,,
& \text{$|D^d u|$--a.e.\ $x\in\Omega$},
\\
\displaystyle\mean{u^-(x)}^{u^+(x)} \Cyl{\bsmall_t}{\nu_u}{x}\, dt
\,,
& \text{$\H^{N-1}$--a.e.\  $x\in J_u$}.
\end{cases}
\end{equation*}
Moreover,
\(
\Cyl{\bsmall_{\ut(x)}}{\nu_u}{x} |\nabla u(x)| = 
{\bsmall(x, u(x))}\cdot {\nabla u(x)}
\)
for $\LLN$-a.e.\ $x\in\Omega$.
\end{theorem}

\begin{proof}
By assumption (iii), 
we can employ \cite[Theorems 2.6 and 3.6]{Anz2}
to deduce that,
for every $t\in\R$,
\begin{gather*}
\frac{d (\bsmall_t, Du)}{d |D u|}
= \Cyl{\bsmall_t}{\nu_u}{x},
\qquad
\text{$|D^d u|$--a.e.\ in $\Omega$},
%\label{f:decdif}
\\
\Trp{\B_t}{J_u} = \Trm{\B_t}{J_u} =
\Cyl{\B_t}{\nu_u}{\cdot},
\qquad
\text{$\H^{N-1}$--a.e.\ in $J_u$}\,.
%\label{f:decjump}
\end{gather*}
%(see \cite[Theorems 2.6 and 3.6]{Anz2}).

The representation of the jump part (i.e.\ of $\Theta$ on $J_u$)
follows directly from \cite[Theorem~4.6]{CD4}
and the simple computation
\[
\Cyl{\B_{u^+(x)}}{\nu_u}{x} - \Cyl{\B_{u^-(x)}}{\nu_u}{x}
= \int_{u^-(x)}^{u^+(x)} \Cyl{\bsmall_t}{\nu_u}{x}\, dt\,.
\]
It remains to prove that
\[
\Theta(\bsmall, Du; \cdot) = 
\Cyl{\bsmall_{\ut}}{\nu_u}{\cdot}
\qquad \text{$|D^d u|$--a.e.\ in $\Omega$}\,.
\]
First, we remark that there exists a Borel set
$N\subset\Omega$, with $|D^d u|(N) = 0$,
such that the limit of cylindrical averages $\Cyl{\bsmall_t}{\nu_u}{\cdot}$
exists for every $x\in\Omega\setminus N$ and every $t\in\R$
(see e.g.\ the proof of \cite[Lemma~4.2]{CD4}).
As a consequence, the map 
$x\mapsto \Cyl{\bsmall_{\ut(x)}}{\nu_u}{x}$
belongs to $L^\infty_{\rm loc}(\Omega, |D^d u|)$.

To simplify the notation, let us denote by
$\mu := (\bsmall(\cdot, u), Du)$ the pairing measure
and let $\mu^d$ denote its diffuse part.
We have to prove that
\[
\frac{d\mu^d}{d|Du|}(x) =
\Cyl{\bsmall_{\ut(x)}}{\nu_u}{x}\,,
\quad
\text{for $|D^d u|$--a.e.\ $x\in \Omega$.}
\]

Let us choose $x\in \Omega$ such that
\begin{itemize}
	\item[(a)]
	\(x\) belongs to the support of \(D^d u\), that is
	\(|D^d u|(B_r(x)) > 0\) for every \(r>0\);
	\item[(b)] 
	there exists the limit
	\(\displaystyle
\lim_{r\downarrow 0}
\frac{\mu^d(B_r(x))}{|D u|(B_r(x))} = \frac{d\mu^d}{d|Du|}(x)\,;
	\)
	\item[(c)]
	\(\displaystyle
\lim_{r\downarrow 0}\frac{|D^j u|(B_r(x))}{|D u|(B_r(x))}=0
	\);

	\item[(d)]
	\(\displaystyle
\lim_{r\downarrow 0}\frac{(\bsmall_{\ut(x)}, Du)(B_r(x))}{|D u|(B_r(x))}
=	\Cyl{\bsmall_{\ut(x)}}{\nu_u}{x}
	\);
\item[(e)]
	\(\displaystyle
\lim_{r\downarrow 0}
\frac{1}{|Du|(B_r(x))}
\int_{B_r(x)} |\ut(y)-\ut(x)|\, d |D^d u|(y) = 0
\).
\end{itemize}
We remark that these conditions are satisfied
for $|D^d u|$-a.e.\ $x\in\Omega$.
In particular, (e) holds since $|D^d u|$-a.e.\ $x\in\Omega$
is a Lebesgue point of $\ut$ with respect to $|Du|$.

Since
\[
\begin{split}
& \left|
\frac{(\bsmall(\cdot, u), Du)(B_r(x))}{|Du|(B_r(x))}
-  \Cyl{\bsmall_{\ut(x)}}{\nu_u}{x}
\right|
\\ & \leq
\left|
\frac{(\bsmall(\cdot, u), Du)(B_r(x))}{|Du|(B_r(x))}
-  \frac{(\bsmall_{\ut(x)}, Du)(B_r(x))}{|Du|(B_r(x))}
\right|
+
\left|
\frac{(\bsmall_{\ut(x)}, Du)(B_r(x))}{|Du|(B_r(x))}
-  \Cyl{\bsmall_{\ut(x)}}{\nu_u}{x}
\right|,
\end{split}
\]
by (d) it is enough to prove that
\begin{equation}\label{f:diffuse}
%\lim_{r\downarrow 0}
I_r := \left|\frac{(\bsmall(\cdot, u), Du)(B_r(x))}{|Du|(B_r(x))}
-  \frac{(\bsmall_{\ut(x)}, Du)(B_r(x))}{|Du|(B_r(x))}
\right| \longrightarrow 0\,,
\qquad\text{as}\ r\searrow 0,
\end{equation}
i.e.
\[
\frac{d(\bsmall(\cdot, u), Du)}{d|Du|}(x)
=  \frac{d(\bsmall_{\ut(x)}, Du)}{d|Du|}(x)\,.
\]

By Lemma~\ref{l:lip},
choosing $\tau = \ut(x)$ and
taking a sequence $\phi_j\in C^\infty_c(B_r(x))$,
$\phi_j(y)\to 1$ in $B_r(x)$, with \(0\leq \phi_j \leq 1\), we get
\[
\begin{split}
I_r \leq {} &
\frac{L}{|Du|(B_r(x))}
\bigg[
\int_{B_r(x)} |\ut(y)-\ut(x)|\, d |D^d u|(y)
\\ & 
+ \int_{B_r(x) \cap J_u} \left(
\int_{u^-(y)}^{u^+(y)} |t - \ut(x)|\, dt
\right)
\, d\H^{N-1}(y)
\bigg]
\\ \leq {} &
\frac{L}{|Du|(B_r(x))}
\int_{B_r(x)} |\ut(y)-\ut(x)|\, d |D^d u|(y)
+
2L \|u\|_\infty \frac{|D^j u|(B_r(x))}{|Du|(B_r(x))}\,.
\end{split}
\]
Finally, by (c) and (e)
we conclude that \eqref{f:diffuse} holds.
\end{proof}

\begin{corollary}
[Integral representation of the pairing functional]
Let $\bsmall$
satisfy assumptions (i)--(iv),
and let $u\in  \BVLloc$.
Then it holds that
\begin{equation*}
\begin{split}
%\label{modulo}
\int_\Omega\phi\, d(\bsmall(x,u),Du)= 
\int_\Omega\phi\mean{u^-}^{u^+} \Cyl{\bsmall_t}{\nu_u}{x}\, dt
\, d|Du|,
\qquad \phi\in C_c(\Omega)\,,
\end{split}
\end{equation*}
where we use the compact notation
\begin{equation*}
\begin{split}
%\label{modulo1}
\mean{u^-}^{u^+} \Cyl{\bsmall_t}{\nu_u}{x}\, dt
|Du|:= 
{} &
\bsmall(x,u)\cdot\nabla u\,\LLN
+ \Cyl{\bsmall_{\ut}}{\nu_u}{x}\, |D^c u|
\\ & 
+ \mean{u^-}^{u^+} \Cyl{\bsmall_t}{\nu_u}{x}\, dt
|D^j u|.
\end{split}
\end{equation*}
\end{corollary}

%%% ELIMINARE?
\begin{remark}
\label{r:53}
As a direct consequence of the above corollary, it holds that
\begin{equation*}
%\label{cantor}
\frac{d(\bsmall(\cdot,u),Du)}{d|D^c u|}(x)
=\Cyl{\bsmall_{\ut}}{\nu_u}{x}
\end{equation*}
for $|D^c u|$-a.e. $x\in\Omega$.
\end{remark}

\begin{theorem}\label{t:repr3}
Let $\bsmall$
satisfy assumptions (i)--(iv), and $u\in\BVLloc$.
Then,
for $|Du|$-a.e.\ $x\in \Omega$,
\begin{equation}
\label{f:repr3}
\Theta(\bsmall, u; x) =
\mean{u^-(x)}^{u^+(x)}
\Trace[]{\bsmall_t}{\partial^* \{u>t\} }(x)
\, dt\,,
\end{equation}
where we use the convention $\mean{a}^a f(t)\, dt := f(a)$.
In particular, 
\begin{equation}
\label{f:reprC}
\Theta(\bsmall, u; x) =
\Trace[]{\bsmall(\cdot,u)}{\partial^* \{u>\widetilde u(x)\} }(x),
\qquad
\text{for $|D^du|$-a.e.\ $x\in \Omega$}.
\end{equation}
\end{theorem}

\begin{proof}
It suffices to prove that, for every Borel set $B\subset \Omega$,
it holds that
\begin{equation}\label{f:repr177}
(\bsmall(\cdot,u), Du) (B) =\int_B
\mean{u^-(x)}^{u^+(x)}
\Trace[]{\bsmall_t}{\partial^* \{u>t\} }(x)
\, dt\,d|Du|.
\end{equation}
For every $t\in\R$ such that $\partial^* \{u > t\}$ is locally of finite perimeter
(and hence for a.e.\ $t\in\R$),
by Corollary \ref{c:paironch} 
we deduce that
%for $\L^1$-a.e.\ $t\in\R$ the following representation formula holds
\begin{equation}\label{f:repr100}
\Theta(\bsmall_t, D\chi_{\{u > t\}}; x)
= \Trace[]{\bsmall_t}{\partial^* \{u > t\} }(x)
\quad
\text{for $\hh$-a.e.}\ x\in \partial^* \{u > t\}.
\end{equation}
Then, by using the coarea formula \eqref{f:coareaB} for the pairing measure, formula \eqref{f:repr100}, Fubini's theorem and Theorem~\ref{coarea}  we get
\[
\begin{split}
& (\bsmall(\cdot,u), Du) (B) 
=\int_{\R}  (\bsmall_t, D\chi_{\{u>t\}}) (B)\, dt
=\int_{\R} \int_{B\cap \partial^*\{u>t\}} \Trace[]{\bsmall_t}{\partial^*\{u>t\}}
	d\hh\, dt
\\
& =\int_{B\setminus J_u}
\Trace[]{\bsmall_t}{\partial^* \{u>\ut(x)\} }
\,d\H^{N-1}+
\int_{B\cap J_u}
\int_{u^-(x)}^{u^+(x)}
\Trace[]{\bsmall_t}{\partial^* \{u>t\} }
\, dt\,d\H^{N-1}
\\
& =
\int_B
\mean{u^-(x)}^{u^+(x)}
\Trace[]{\bsmall_t}{\partial^* \{u>t\} }
\, dt\,d|Du|\,,
\end{split}
\]
where in the last equality we have used that $|D^ju|=(u^+-u^-)\H^{N-1}\res J_u$,
so that \eqref{f:repr177} is proved.
\end{proof}

\section{Lower semicontinuity of the pairing}
\label{ss:lower}

In this section, by using the nonautonomous chain rule formula \eqref{lolo} for the divergence, we study the lower semicontinuity with respect to the 
$L^1$ %weak$^*$ 
convergence 
of the functionals
$F, G^+\colon \BVL\to\R$ defined by
\[
F(u):= \int_\Omega d|(\bsmall(x, u),Du)|\,,
\quad
G(u):= \int_\Omega d(\bsmall(x, u),Du)\,,
\quad
G^+(u):= \int_\Omega d(\bsmall(x, u),Du)^+\,.
%\qquad u\in \BVL\,.
\]

We start by proving the following continuity result (see \cite{dcl} for the
analogous result in $W^{1,1}$).

\begin{proposition}\label{p:cont}
Let $\bsmall$
satisfy assumptions (i)--(iv),
let $\varphi\in C^1_c(\Omega)$ be a fixed test function,
and let $G_\varphi\colon \BVL\to\R$ be the functional defined by
\[
G_\varphi(u) :=
\pscal{(\bsmall(x, u), Du)}{\varphi}= \int_\Omega \varphi\, d (\bsmall(x, u),Du)\,,
\qquad u\in\BVL.
\]
Let $(u_j) \subset  \BVL$ and $u\in\BVL$ satisfy
one of the following assumptions:
\begin{itemize}
\item[(a)]
the sequence $(u_j)$ converges strongly to $u$ in $L^1_{\rm loc}(\Omega)$, and
\begin{equation*}
%\label{equib}
\Lambda := \sup_j \|u_j\|_{L^\infty(\Omega)} < +\infty\,;
\end{equation*}
\item[(b)]
$\sigma\in L^\infty_{\rm loc}(\Omega)$
and $(u_j)$ converges strongly to $u$ in $L^1_{\rm loc}(\Omega)$;
\item[(c)]
the function $\sigma$ defined in~\eqref{f:sigma} satisfies the stronger condition
$\sigma\in L^N_{\rm loc}(\Omega)$,
and
$(u_j)$ converges to $u$ weakly$^*$ in $BV(\Omega)$.
\end{itemize}
If either (a) or (b) or (c) holds, then
\begin{equation*}
%\label{continu}
\lim_{j\to +\infty} G_\varphi(u_j) = G_\varphi(u).
\end{equation*}
%
%Then,  
%for every sequence $(u_j) \subset  \BVL$ converging to
%$u\in \BVL$ in the $L^1$ sense,
%and satisfying 
%\begin{equation*}
%%\label{equib}
%\Lambda := \sup_j \|u_j\|_{L^\infty(\Omega)} < +\infty,
%\end{equation*}
%it holds that
%\begin{equation}\label{continu}
%\lim_{j\to +\infty} G_\varphi(u_j) = G_\varphi(u).
%\end{equation}
%Moreover, by assuming, instead of (iv), the stronger condition $\sigma\in L^N_{\rm loc}(\Omega)$ (respectively $\sigma\in L^\infty_{\rm loc}(\Omega)$), the continuity \eqref{continu} holds 
%if $(u_j)$ converges to $u$ weakly$^*$ in $BV(\Omega)$
%%for every sequence $(u_j) \subset  \BVL$ weak$^*$-converging to $u\in \BVL$ 
%(respectively strongly in $L^1_{\rm loc}(\Omega)$).
\end{proposition}

\begin{proof}
Assume that (a) holds.
Using \eqref{f:bxurepr} we have that
\[
G_\varphi(u_j) - G_\varphi(u) =
-\int_\Omega \varphi(x) \int_{u(x)}^{u_j(x)} \Div_x \bsmall_t(x)\, dt\, dx
- \int_\Omega \int_{u(x)}^{u_j(x)} \bsmall_t(x)\cdot \nabla\varphi(x)\, dt\, dx.
\]
Since $\bsmall$ is a locally bounded vector field, the second integral converges to $0$
by the Lebesgue Dominated Convergence theorem.
The first integral can be written as
\begin{equation}\label{f:fi}
\iint_{K\times[-\Lambda, \Lambda]} \sign(u(x) - u_j(x))
\chi_{D_j}(x,s)\, \Div_x\bsmall_s(x)\, dx\, ds,
\end{equation}
where $K\subset\Omega$ is the support of $\varphi$ and 
$D_j \subset\Omega\times [-\Lambda, \Lambda]$ is the set of pairs $(x,s)$ such that
$s$ belongs to the segment of endpoints $u(x)$ and $u_j(x)$.
Since
\[
\left|
\chi_{D_j}(x,s) \varphi(x)\, \Div_x \bsmall_s(x)
\right|
\leq
\|\varphi\|_{L^\infty(\Omega)}
\, |\Div_x \bsmall_s(x)|
\leq
\|\varphi\|_{L^\infty(\Omega)}
\, \sigma(x)
\in L^1(K\times[-\Lambda, \Lambda]),
\]
the integral \eqref{f:fi} converges to $0$ by the Lebesgue Dominated Convergence theorem.

%Let us prove the second part of the theorem.
Assume now that assumption (c) holds.
If the sequence $(u_j) \subset  \BVL$ weak$^*$-converges to
$u\in \BVL$, by the Poincaré inequality
(see \cite[Remark~3.50]{AFP})
 we have that $(u_j)$ weakly converges to
$u$ in $L^{\frac N{N-1}}(\Omega)$.
Since
\[
\left|\int_\Omega \varphi(x) \int_{u(x)}^{u_j(x)} \Div_x \bsmall_t(x)\, dt\, dx\right|\leq
\|\varphi\|_\infty
\int_K |u_j(x)-u(x)|  \sigma(x)\, dx,
\]
then, if $\sigma\in L^N_{\rm loc}(\Omega)$,
 the integral on the right-hand side converges to $0$. 

We arrive at the same conclusion also if assumption (c) holds, that is, if $(u_j)$ converges to
$u$ strongly in $L^1_{\rm loc}(\Omega)$ and $\sigma\in L^\infty_{\rm loc}(\Omega)$.
\end{proof}

\begin{remark}
In \cite{DM} Dal Maso proved the lower semicontinuity
of integral functionals with coercive integrands and he showed, by exploiting Aronszajn's example, that
this result is sharp, in the sense that, in general, 
the coercivity assumption cannot be dropped. 
Indeed, Dal Maso constructed a continuous function
$\omega\colon\Omega\to \R$, where $\Omega=(0,1)\times (0,1)$ and
$x=(x_1,x_2)$, and a sequence of functions $\{u_n\}$
 converging to $u(x)=x_2$ in $L^\infty(\Omega)$, such that
\[
\int_\Omega |(\sin \omega (x),\cos \omega(x))\cdot\nabla u(x))|\,dx
>\liminf_{n\to\infty}
\int_\Omega |(\sin \omega(x),\cos \omega(x))\cdot\nabla u_n(x))|\,dx.
\]
Let us remark that the integrand $|\bsmall(x)\cdot\xi|$ of Dal Maso's example does not satisfy our condition $\Div\bsmall\in L^1$.
\end{remark}

\begin{theorem}\label{p:lsc}
[Lower semicontinuity of $F$ and $G^+$]
Let $\bsmall$
satisfy assumptions (i)--(iv). 
Let $(u_j) \subset  \BVL$ and $u\in\BVL$ satisfy
one of the assumptions (a), (b) or (c) listed
in Proposition~\ref{p:cont}.
Then
\[
F(u) \leq \liminf_{j\to +\infty} F(u_j),
\qquad
G^+(u) \leq \liminf_{j\to +\infty} G^+(u_j).
\]
%Then the functionals $F, G^+$ are lower semicontinuous on $\BVL$ 
%with respect to the %weak$^*$
%$L^1$ convergence. 
\end{theorem}

\begin{proof}
Let us define the auxiliary functionals
$H, H^+\colon\BVL\to\R$ by
\[
H(u):= -\int_\Omega d(\bsmall(x, u),Du)\,,
\qquad
H^+(u):= \int_\Omega d(-(\bsmall(x, u),Du))^+\,.
\]
Since $F(u)=G^+(u)+H^+(u)$,
it suffices to prove that $G^+(u)$ and $H^+(u)$ are lower semicontinuous on $\BVL$ with respect to the %weak$^*$
$L^1$ convergence. 
We shall prove the claim only for $G^+$, being the proof for $H^+$ similar.

Let $\Phi$ denote the  set of all functions $\varphi\in C^1_c(\Omega)$ such that $0\leq\varphi\leq 1$.
Since
\begin{equation*}
%\label{fff266}
G^+(v) = \sup_{\varphi\in\Phi}\ G_\varphi (v),
\qquad
\forall v\in \BVL,
\end{equation*}
the conclusion follows from Proposition~\ref{p:cont}.
\end{proof}

%%%%%%%%%%%%%%%%%%%%%%%%%%%%%%%%%%%%%%%%%%%%%%%%%%%%%%%%%%%%%%%%%%%%%%%%%%%%%%%%%%%

\section{Pairing as relaxed functional}
\label{s:relax}

For every function $\varphi\in C^1_c(\Omega)$ 
and for every open set $A\subseteq\Omega$
let us consider  the functional 
$F_\varphi(\cdot, A):\BVLloc\to\R\cup\{+\infty\}$ 
defined in~\eqref{f:Fphi}.
%by
%\[
%F^\varphi(u,A):=
%\begin{cases}
%\displaystyle\int_A \varphi\,\bsmall(x,u)\cdot\nabla u\,dx
%& \text{$u\in W^{1,1}_{{\rm loc}}(\Omega)$}\cap L_{{\rm loc}}^\infty(\Omega),
%\\
%+\infty
%,
%& u\in (BV_{{\rm loc}}(\Omega)\
%\setminus W_{{\rm loc}}^{1,1}(\Omega))\cap L_{{\rm loc}}^\infty(\Omega).
%\end{cases}
%\]
Moreover, for every $u\in\BVLloc$
we consider its relaxation
$\rel{F_\varphi}(u, A)$ 
with respect to the weak$^*$ convergence in $\BVLloc$, defined in \eqref{relax1},
and with respect to the $L^1_{\rm loc}(\Omega)$ convergence, defined by
%\[
%\rel{F^\varphi}(u,A) :=
%\inf_{\{u_n\}} 
%\left\{
%\liminf_{n\to +\infty}
%F^\varphi(u_n, A)\colon
%u_n\in W^{1,1}_{\text{loc}}(A),\
%u_n\weakstarto u\ \text{in}\ BV_{{\rm loc}}(A)
%\right\}
%\,,
%\]
\[
\rel{F^1_\varphi}(u, A) :=
\inf_{\{u_n\}} 
\left\{
\liminf_{n\to +\infty}
F_\varphi(u_n, A)\colon
u_n\in W^{1,1}_{\text{loc}}(\Omega)\cap L^\infty_{\rm loc}(\Omega),\
u_n\to u\ \text{in}\ L^1_{\rm loc}(\Omega)
\right\}
\,.
\]

\begin{theorem}
\label{t:limsupG}
[Integral representation of the relaxed functionals of $F_\varphi$]
Let $\bsmall$
satisfy assumptions (i)--(iv).
Then for every %$u\in  BV(\Omega)$
$u\in  \BVLloc$ 
and for every 
%$A\in\A(\Omega)$
open set $A\subseteq\Omega$,
if we assume $\sigma\in L^N_{\rm {loc}}(\Omega)$, then it holds that
\begin{equation*}
%\label{t:generalenuovo1}
\rel{F_\varphi}(u, A)=\int_A \varphi\, d (\bsmall(x,u),Du)
\end{equation*}
whereas, if we assume $\sigma\in L^\infty_{\rm {loc}}(\Omega)$, then it holds that
\begin{equation*}
%\label{t:generalenuovo1}
\rel{F^1_\varphi}(u, A)=\int_A\varphi\, d (\bsmall(x,u),Du).
\end{equation*}
\end{theorem}

\begin{proof}
We shall prove the result only for $\rel{F_\varphi}$.
Thanks to the continuity results proved
in Proposition~\ref{p:cont} 
%Theorem~\ref{p:lsc}
and the argument in \cite[Theorem 1.3]{FL1},
it is enough to prove the following two inequalities:
\begin{itemize}
\item[(J)] for $\H^{N-1}$-a.e. $x_0\in J_u$, it holds that
\[
\frac{d\rel{F_\varphi}(u,\cdot)}{d\H^{N-1}\ristretto{J_u}} (x_0)
\leq \varphi(x_0)\int_{u^-(x_0)}^{u^+(x_0)}\!
\Cyl{\bsmall(\cdot,t)}{\nu_u(x_0)}{x_0}\,dt\,;
\]
\item[(C)] for $\vert D^c u\vert$-a.e. $x_0\in \Omega$, it holds that
\[
\frac{d\rel{F_\varphi}(u,\cdot)}{d\vert D^c u\vert} (x_0)
\leq  \varphi(x_0)\Cyl{\bsmall(\cdot,\preciso{u}(x_0))}{\nu_u(x_0)}{x_0}\,.
\] 
\end{itemize}
%The next two propositions are devoted to the proof of
%(C) and (J) respectively.
Since both results are of local nature,
it is not restrictive to assume that $\Omega = \R^N$
and that
$u\in BV(\R^N)\cap L^\infty(\R^N)$.
Moreover, to simplify the notation we denote
\(\mu:=(\bsmall(\cdot, u),Du)\).

\bigskip
\noindent
\textbf{Proof of (J).}
By the definition of relaxed functional we have that
\[
\begin{split}
& \frac{d\rel{F_\varphi}(u,\cdot)}{d\H^{N-1}\ristretto{J_u}} (x_0)
= \lim_{r\searrow 0} \frac{\rel{F_\varphi}(u, B_r(x_0))}{\H^{N-1}\ristretto{J_u}(B_r(x_0))}
%{\omega_{N-1} r^{N-1}}
\\ & \quad \leq 
\lim_{r\searrow 0} 
\liminf_{\varepsilon\searrow 0}
\frac{1}{\H^{N-1}\ristretto{J_u}(B_r(x_0))}
%\frac{1}{\omega_{N-1} r^{N-1}}
\int_{B_r(x_0)} \varphi(y)\, \bsmall(y, \rho_\varepsilon\ast u (y))\cdot
\nabla(\rho_\varepsilon\ast u)(y)\, dy\,.
\end{split}
\]
As $\varepsilon\to 0^+$,
the integral above converges to
$\int_{B_r(x_0)} \varphi \, d\mu$
(see the proof of Theorem~4.3 in \cite{CD4}, where this convergence
is stated in formula (38)).
Hence,
\[
\begin{split}
\frac{d\rel{F_\varphi}(u,\cdot)}{d\H^{N-1}\ristretto{J_u}} (x_0)
& \leq
\lim_{r\searrow 0} 
\frac{1}{\H^{N-1}\ristretto{J_u}(B_r(x_0))}
%\frac{1}{\omega_{N-1} r^{N-1}}
\int_{B_r(x_0)} \varphi\, d\mu
\\ & =
\varphi(x_0) \,
\lim_{r\searrow 0}
\frac{|Du|(B_r(x_0))}{\H^{N-1}\ristretto{J_u}(B_r(x_0))}
%{\omega_{N-1} r^{N-1}}
\cdot
 \frac{\mu(B_r(x_0))}{|Du|(B_r(x_0))}
\\ & =
\varphi(x_0) \, [u^+(x_0) - u^-(x_0)] \, \Theta(\bsmall, u, x_0)\,,
\end{split}
\]
so that (J) follows from 
%\eqref{f:repr} in
Theorem~\ref{t:repr2}.

\bigskip
\noindent
\textbf{Proof of (C).}
Reasoning as in the proof of (J) above, we have that
\[
\begin{split}
\frac{d\rel{F_\varphi}(u,\cdot)}{d |D^c u|} (x_0)
& = \lim_{r\searrow 0} \frac{\rel{F_\varphi}(u, B_r(x_0))}{|D^c u|(B_r(x_0))}
\\ & \leq 
\varphi(x_0) \,
\lim_{r\searrow 0}
\frac{\mu(B_r(x_0))}{|D^c u| (B_r(x_0))}
= 
\varphi(x_0)\, \Theta(\bsmall, u, x_0)\,,
\end{split}
\]
and the conclusion follows from 
Remark~\ref{r:53}.
%Theorem~\ref{t:repr2}.
\end{proof}

\medskip
\textbf{Acknowledgments.}
The authors want to thank two anonymous referees for carefully reading the manuscript and for giving constructive comments which helped improve the quality of the paper.

%% % % % % % % % % % % % % %
%\bibliographystyle{mybst}
%\bibliography{Ricerca,Graziano-MR}
%\end{document}

\def\cprime{$'$}
% \bib, bibdiv, biblist are defined by the amsrefs package.

\end{document}